\documentclass[a4paper,10pt]{amsart}
\usepackage{amsfonts, amsmath, amsthm, amssymb}

\usepackage[cp1251]{inputenc}
\usepackage[russian, english]{babel}

\usepackage{xcolor,graphics}

\newcommand{\E}{\mathbb{E}}
\newcommand{\R}{\mathbb{R}}
\newcommand{\bx}{\mathbf{x}}

\newcommand{\Vol}{\mathop{\mathrm{Vol}}\nolimits}

\newtheorem{theorem}{Theorem}[section]

\newtheorem{lemma}[theorem]{Lemma}
\newtheorem{corollary}[theorem]{Corollary}
\newtheorem{proposition}[theorem]{Proposition}

\title[Correlations of algebraic numbers]{Correlations between \\ real conjugate algebraic numbers}

\author[F.~G\"otze]{Friedrich G\"otze}
\address{Friedrich G\"otze, Faculty of Mathematics,
Bielefeld University,
P. O. Box 10 01 31,
33501 Bielefeld, Germany}
\email{goetze@math.uni-bielefeld.de}

\author[D.~Kaliada]{Dzianis Kaliada}
\address{Dzianis Kaliada, Institute of Mathematics, National Academy of Sciences of Belarus, 220072 Minsk, Belarus}
\email{koledad@rambler.ru}

\author[D.~Zaporozhets]{Dmitry Zaporozhets}
\address{Dmitry Zaporozhets\\
St.\ Petersburg Department of
Steklov Institute of Mathematics,
Fontanka~27,
 191011 St.\ Petersburg,
Russia}
\email{zap1979@gmail.com}

\thanks{Supported by CRC 701, Bielefeld University (Germany).}

\begin{document}

\begin{abstract}
For $B\subset\R^k$ denote by $\Phi_k(Q;B)$ the number of ordered $k$-tuples in $B$ of real conjugate algebraic numbers of degree $\leq n$ and naive height $\leq Q$. We show that
$$
\Phi_k(Q;B) = \frac{(2Q)^{n+1}}{2\zeta(n+1)} \int_{B} \rho_k(\bx)\,d\bx + O\left(Q^n\right),\quad Q\to \infty,
$$
where the function $\rho_k$ will be given explicitly. If $n=2$, then an additional factor $\log Q$ appears in the reminder term.
\end{abstract}

\maketitle

\section{Introdution}
Baker and Schmidt~\cite{BS70} proved that the set of algebraic numbers of degree at most $n$ forms a \emph{regular system}: there exists a constant $c_n$ depending on $n$ only such that for any interval $I\subset\R^1$ and for all sufficiently large $Q\in\mathbb N$  there exist at least
\[
c_n|I|\, Q^{n+1} / (\log Q)^{3n(n+1)}
\]
algebraic numbers $\alpha_1,\dots,\alpha_l$ of degree at most $n$ and height at most $Q$ satisfying
\[
|\alpha_i-\alpha_j|\geq (\log Q)^{3n(n+1)} / Q^{n+1},\quad 1\leq i<j\leq l.
\]
Later Beresnevich~\cite{vB99} showed that the logarithmic factors can be omitted.

Beresnevich, Bernik, and G\"otze~\cite{BBG10} obtained the following result about the distribution of distances between conjugate algebraic numbers.
Let $n\ge 2$ and $0 < w \le \frac{n+1}{3}$. Then for all sufficiently large $Q$ and any interval $I\subset[-\frac12, \frac12]$ there exist at least $\frac12 Q^{n+1-2w}|I|$ real algebraic numbers $\alpha$ of degree $n$ and height $H(\alpha)\asymp_n Q$ having a real conjugate $\alpha^*$ such that $|\alpha - \alpha^*| \asymp_n Q^{-w}$.

The simplest example of algebraic numbers are the rational numbers. Appropriately ordered,  they form \emph{the Farey sequences}. It is well-known that they are equidistributed in $[0,1]$, see \cite{mM1949} for an elementary proof and~\cite{fD99} for a deeper discussion of the subject. For the history of the problem we refer the reader to~\cite{CZ2003}.

In 1971 Brown and Mahler \cite{BM71} introduced a natural generalization of the Farey sequences:
the Farey sequence of degree $n$ and order $Q$ is the set of all real roots of integral polynomials of degree $n$ and height at most $Q$.
The distribution of the generalized Farey sequences has been investigated in~\cite{dK14} (see also~\cite{dK13}, \cite{dK15} for the case $n=2$).

Namely, fix $n\geq2$ and consider an arbitrary interval $I\subset\mathbb{R}$. Denote by $\Phi(Q;I)$ a number of algebraic numbers $\alpha\in I$ of degree at most $n$ and height at most $Q$. Then we have that
\begin{equation}\label{2340}
\Phi(Q;I)=\frac{(2Q)^{n+1}}{2\zeta(n+1)}\int_I\rho(x)\,dx+ O\left(Q^n \log^{l(n)} Q\right),\quad Q\to \infty,
\end{equation}
where $\zeta(\cdot)$  denotes  the the Riemann zeta function and $l(n)$ is defined by
\begin{equation}\label{2304}
l(n) =  \begin{cases}
   1, & n=2,\\
   0,  & n\geq 3.
 \end{cases}
\end{equation}
The limit density $\rho$ is given by the formula
\begin{equation}\label{1848}
\rho(x)=2^{-n-1}\int_{D_x}\left|\sum_{j=1}^n jt_jx^{j-1}\right|\,dt_1\dots dt_n,
\end{equation}
where the domain of integration $D_x$ is defined by
\[
D_x=\left\{(t_1,\dots,t_n)\in\R^n\,:\,\max_{1\leq k\leq n}|t_k|\leq1, \ |t_nx^n+\dots+t_1 x|\leq1 \right\}.
\]
If $x\in[-1+1/\sqrt2, 1-1/\sqrt2]$, then~\eqref{1848} can be simplified as follows:
\[
\rho(x)=\frac{1}{12}\left(3+\sum_{k=1}^{n-1}(k+1)^2x^{2k}\right).
\]

The function $\rho$ coincides with the density of the real zeros of the random polynomial
\begin{equation}\label{2302}
G(x)=\xi_nx^n+\xi_{n-1}x^{n-1}+\dots+\xi_1x+\xi_0,
\end{equation}
where $\xi_0,\xi_1,\dots,\xi_{n}$ are independent random variables uniformly distributed on $[-1,1]$ (see, e.g.,~\cite{dZ05}). It means that for any Borel subset $B\subset\R^1$,
$$
\E N(G,B)=\int_B \rho(x)\,dx,
$$
where $N(G,B)$ denotes the number of zeros of $G$ lying in $B$. The real zeros of $G$ can be considered as a random point process. Its distribution  can be described by its $k$-point correlation functions $\rho_k(x_1,\dots,x_k), k=1,2,\dots,n$ (also known as joint intensities; see Section~\ref{1513} for definition). The one-point correlation function $\rho_1$ coincides with the density $\rho$. The explicit formula for $\rho_k$ has been obtained in~\cite{GKZ15b} (see Section~\ref{1513} for details).

The aim of this paper is to show that, like in the case $k=1$, the correlation functions $\rho_k$ are closely related with the joint distribution of real conjugate algebraic numbers.

Note that for $k=n$ we obtain the joint distribution of totally real algebraic numbers. These numbers (in particular, fields formed by them) are of great interest and possess some interesting properties. For example, lattices built by them are very well distributed in parallelepipeds, see \cite{mS1989}.

We also mention that there are a number of papers where total number of $k$-vectors whose coordinates form a field extension of degree $n$ over some base number field is considered. See e.g. \cite{MV07} or \cite{fB2014} for some interesting results and references. In this area the multiplicative Weil height is usually used to measure such vectors.

\section{Notations and main result}\label{1513}
Let us start with some notation. Fix some positive integer $n\geq 2$ and $k\leq n$. Denote
$$\bx=(x_1,\dots, x_k)\in\R^k.
$$

We use the following notation for the elementary symmetric polynomials:
$$
\sigma_i(\bx) :=
\left\{
  \begin{array}{ll}
    1, & \hbox{if}\quad i=0, \\
    \sum_{1\le j_1 < \dots < j_i\le k} x_{j_1} x_{j_2} \dots x_{j_i}, & \hbox{if}\quad 1\leq i\le k, \\
    0, & \hbox{otherwise.}
  \end{array}
\right.
$$

Denote by $\mathcal{P}(Q)$ the class of all integral polynomials of degree at most $n$ and height at most $Q$. The cardinality of this class is $(2Q+1)^{n+1}$.

Recall that an integral polynomial is called \emph{prime}, if it is irreducible over $\mathbb{Q}$, primitive (the greatest common divisor of its coefficients equals 1), and its leading coefficient is positive. Let $\mathcal{P^*}(Q)$ be the class of all prime polynomials from $\mathcal{P}(Q)$.

The \emph{minimal polynomial} of an algebraic number $\alpha$ is a prime polynomial such that $\alpha$ is a root of this polynomial.

For a Borel subset $B\subset\R^k$ denote by $\Phi_k(Q;B)$ the number of ordered $k$-tuples $(\alpha_1,\alpha_2,\dots,\alpha_k)\in B$ of \emph{distinct} real numbers such that for some $p\in\mathcal{P^*}(Q)$ it holds
$$
p(\alpha_1)=\dots=p(\alpha_k)=0.
$$
Essentially $\Phi_k(Q;B)$ denotes the number of ordered $k$-tuples in $B$ of conjugate algebraic numbers of degree at most $n$ and height at most $Q$.

Given a function $g:\R^1\to\R^1$ and a Borel subset $B\subset\R^k$ denote by $N_{k}(g,B)$ the number of ordered $k$-tuples $(x_1,x_2,\dots,x_k)\in B$ of \emph{distinct} real numbers such that
$$
g(x_1)=\dots=g(x_k)=0.
$$
For any algebraic number its minimal polynomial is prime, and any prime polynomial is a minimal polynomial for some algebraic number. Therefore we have that
\begin{equation}\label{1838}
\Phi_k(Q;B)= \sum_{p\in\mathcal{P}^*(Q)} N_{k}(p,B).
\end{equation}
Applying Fubini's theorem to the right hand side we obtain
\begin{equation}\label{022}
\Phi_k(Q;B)= \sum_{m=0}^\infty m\cdot\#\{p\in\mathcal{P}^*(Q)\,:\,N_{k}(p,B)=m\}.
\end{equation}
Since $N_{k}(p,B)\leq n!/(n-k)!$, the sum in the right hand side is finite.

Now we are ready to state our main result.
\begin{theorem}\label{2034}
Let $B$ be a region in $\mathbb{R}^k$ with boundary consisting of a finite number of algebraic surfaces. Then
\begin{equation}\label{eq-Psi}
\Phi_k(Q;B) = \frac{(2Q)^{n+1}}{2\zeta(n+1)} \int_{B} \rho_k(\bx)\,d\bx + O\left(Q^n  \log^{l(n)} Q \right),\quad Q\to \infty.
\end{equation}
Here the function $\rho_k$ is given by the formula
\begin{equation}\label{1451}
\rho_k(\mathbf{x}) =2^{-n-1} \prod_{1\le i < j \le k} |x_i - x_j|\int\limits_{D_{\bx}} \prod_{i=1}^k \left|\sum_{j=0}^{n-k} t_j x_i^j\right| \, dt_0\dots dt_{n-k},
\end{equation}
where the domain of integration $D_{\bx}$ is defined by
$$
D_{\bx}=\left\{(t_0,\dots, t_{n-k}) \in \R^{n-k+1}\,:\,\max_{0\leq i \leq n}\left|\sum_{j=0}^{n-k} (-1)^j \sigma_{k-i+j}(\bx)t_j\right|\leq 1\right\}.
$$
\end{theorem}
The implicit big-O-constant in~\eqref{eq-Psi} depends on $n$, the number of the algebraic surfaces and their maximal degree only. The proof of Theorem~\ref{2034} is given in Section~\ref{1522}.
\begin{corollary}
The case $k=1$ implies~\eqref{2340}.
\end{corollary}
\begin{corollary}\label{1520}
If $k=n$, then~\eqref{1451} can be simplified as follows:
$$
\rho_n(\mathbf{x}) = \frac{2^{-n}}{(n+1)}\left(\frac{1}{\max_{0\le i\le n} |\sigma_i(\mathbf{x})|}\right)^{n+1} \prod_{1\le i < j \le n} |x_i - x_j|.
$$

\end{corollary}

It has been shown in~\cite[Section~3]{GKZ15b} that the function $\rho_k$ defined in~\eqref{1451} is a $k$-point correlation function of real zeros of the random polynomial $G$ defined in~\eqref{2302}. It  means that
for any Borel subset $B\subset\R^k$,
\begin{equation}\label{2053}
\mathbb{E}N_k(G,B)=\int_B\rho_k(\bx)\,d\bx.
\end{equation}

Let us derive several properties of $\rho_k$.

\begin{proposition}
a) For any permutation $s$ of length $n$,
$$
\rho_k(x_{s(1)}, x_{s(2)}, \dots, x_{s(k)}) = \rho_k(\bx).
$$
b) For all $\mathbf{x}\in\mathbb{R}^k$,
$$
\rho_k(- \mathbf{x}) = \rho_k(\mathbf{x}).
$$
c) For all $\bx\in\R^k$ with non-zero coordinates,
$$
\rho_{k}(x_1^{-1}, x_2^{-1}, \dots, x_k^{-1}) = \rho_{k}(\bx) \prod_{i=1}^k x_i^2.
$$
\end{proposition}
\begin{proof}
The first and the second properties are trivial. To prove the last one, note that for any integral irreducible polynomial $g(z)$ of degree $n$, the polynomial $z^n g(z^{-1})$ is also irreducible and has the same degree and height. Therefore for any Borel set $B\subset \mathbb{R}^k$ which does not contain points with zero coordinates we have
\[
\Phi_k(Q;B^{-1}) = \Phi_k(Q;B),
\]
where $B^{-1}$ is defined as
\[
B^{-1} := \left\{(x_1^{-1}, x_2^{-1}, \dots, x_k^{-1}) : (x_1, x_2, \dots, x_k)\in B\right\}.
\]
Letting $Q$ tend to infinity, we obtain from~\eqref{eq-Psi} that
\[
\int_B \rho_k(\mathbf{x}) \,d\mathbf{x} = \int_{B^{-1}} \rho_k(\mathbf{x}) \,d\mathbf{x}.
\]
Making the substitution $(x_1,\dots, x_k)\to (x_1^{-1}, \dots, x_k^{-1})$, we obtain
\[
\int_B \rho_k(\mathbf{x}) \,d\mathbf{x} = \int_B \left(\prod_{i=1}^k x_i^{-2}\right) \rho_k(x_1^{-1}, x_2^{-1}, \dots, x_k^{-1}) \,d\mathbf{x}.
\]
Since the class of sets $B$ is large enough, the third property follows.
\end{proof}

\section{Proof of Theorem~\ref{2034}}\label{1522}
For a Borel set $A\subset\R^n$ denote by $\lambda^*(A)$ the number of points in $A$ with coprime integer coordinates.

Consider a set $A_m\subset[-1,1]^{n+1}$ consisting of all points $(t_0,\dots,t_n)\in[-1,1]^{n+1}$ such that
$$
N_k(t_nx^n+\dots+t_1x+t_0,B)=m.
$$
Then the number of primitive polynomials  $p\in\mathcal{P}(Q)$ such that $N_k(p,B)=m$ is equal to $\lambda^*(QA_m)$. Hence it follows from the definition of a prime polynomial that
\begin{equation}\label{020}
\left|\#\{p\in\mathcal{P}^*(Q)\,:\,N_{k}(p,B)=m\} - \frac12 \lambda^*(QA_m)\right|\leq R_Q,
\end{equation}
where $R_Q$ denotes the number of reducible  polynomials (over $\mathbb Q$) from $\mathcal{P}_Q$. Note that the factor $1/2$ arises  because prime polynomials have positive leading coefficient.
It is known (see~\cite{bW36}) that
\begin{equation}\label{021}
R_Q=O\left(Q^n\log^{l(n)}Q\right),\quad Q\to\infty.
\end{equation}
Combining~\eqref{020} and~\eqref{021} with~\eqref{022}, we obtain
\begin{equation}\label{245}
\Phi_k(Q;B)= \frac12 \sum_{m=0}^\infty m \lambda^*(Q A_m)+O\left(Q^n\log^{l(n)}Q \right),\quad Q\to\infty.
\end{equation}

To estimate $\lambda^*(QA_m)$, we need the following lemma.

\begin{lemma}\label{lm-prim-pnt}
Consider a region $A\subset\mathbb{R}^d$, $d\geq2,$ with boundary consisting of a finite number of algebraic surfaces. Then
\begin{equation}\label{1318}
\lambda^*(QA)=\frac{\Vol(A)}{\zeta(n)}Q^d+O\left(Q^{d-1}\log^{l(d)}Q\right),\quad Q\to\infty,
\end{equation}
where the implicit constant in the big-O-notation depends on $d$, the number of the algebraic surfaces and their maximal degree only.
\end{lemma}
\begin{proof}
The results of this type are well-known, see, e.g., the  classical monograph by Bachmann \cite[pp. 436--444]{pB94} (in particular, formulas~(83a) and~(83b) on pages 441--442). For the proof of Lemma~\ref{lm-prim-pnt}, see~\cite{GKZ15}.
\end{proof}


Since the boundary of $B$ consists of a finite number of algebraic surfaces, the same is true for $A_m$. Hence it follows from Lemma~\ref{lm-prim-pnt} that
\[
\lambda^*(QA_m)=\frac{\Vol(A_m)}{\zeta(n+1)}Q^{n+1}+O\left(Q^n\right),\quad Q\to\infty,
\]
which together with~\eqref{245} implies
\begin{equation}\label{1140}
\Phi_k(Q;\Omega)= \frac{Q^{n+1}}{2\zeta(n+1)}\sum_{m=0}^\infty m\Vol(A_m) +O\left(Q^n \log^{l(n)} Q\right),\quad Q\to\infty.
\end{equation}

To calculate $\sum_{m=0}^\infty m\Vol(A_m)$, note that
\[
\Vol(A_m) = 2^{n+1}\,\mathbb{P}(N_k(G,B)=m),
\]
where $G$ is the random polynomial defined in~\eqref{2302}. Hence
\begin{equation}\label{1141}
\sum_{m=0}^\infty m\Vol(A_m) = 2^{n+1}\, \mathbb{E} N_k(G,B).
\end{equation}
Applying~\eqref{2053} finishes the proof.

\bigskip

{\bf Acknowledgments.} The authors are grateful to Vasily Bernik for many useful discussions.

\bibliographystyle{plain}
\bibliography{corrf4}


\end{document}